\documentclass[11pt]{amsart}
\usepackage[english]{babel}
\usepackage[latin1]{inputenc}
\usepackage[dvips,final]{graphics}
\usepackage{amsmath,amsfonts,amssymb,amsthm,amscd,array,stmaryrd,mathrsfs}
\usepackage{pstricks}
\usepackage{url}
 \usepackage[all]{xy}
\usepackage{graphicx}
\usepackage{color}              
\usepackage{epsfig}
\usepackage{psfrag}
\usepackage{epstopdf}

\usepackage{textcomp}

\let\ssection=\section
\renewcommand{\section}{\setcounter{equation}{0}\ssection}

\theoremstyle{plain}
\newtheorem{thm}{Theorem}
\newtheorem{lem}{Lemma}[section]

\newtheorem{prop}[lem]{Proposition}

\theoremstyle{definition}

\newcommand{\R}{\mathbb{R}}

\newcommand{\C}{\mathbb{C}}
\newcommand{\RP}{\mathbb{RP}}
\newcommand{\CP}{\mathbb{CP}}
\newcommand{\HP}{\mathbb{HP}}
\newcommand{\bbH}{\mathbb{H}}

\newcommand{\bbO}{\mathbb{O}}

\newcommand{\Id}{\textup{Id}}
\newcommand{\e}{\varepsilon}

\def\e{\varepsilon}

\def\r{\rho}

\begin{document}

\title[On fibrations with flat fibers]{On fibrations with flat fibers}

\author{
Valentin Ovsienko
\and
Serge Tabachnikov}

\address{
Valentin Ovsienko,
CNRS,
Institut Camille Jordan,
Universit\'e Claude Bernard Lyon~1,
43 boulevard du 11 novembre 1918,
69622 Villeurbanne cedex,
France}
\email{ovsienko@math.univ-lyon1.fr}

\address{
Serge Tabachnikov,
Department of Mathematics,
Pennsylvania State University, 
University Park, PA 16802, USA
}
\email{tabachni@math.psu.edu}

\date{}

\begin{abstract}
We describe pairs $(p,n)$ such that $n$-dimensional affine space is fibered by
pairwise skew $p$-dimensional affine subspaces.
The problem is closely related with
the theorem of Adams on vector fields on spheres and
the Hurwitz-Radon theory of composition of quadratic forms.
\end{abstract}

\maketitle

\thispagestyle{empty}

\section{Introduction}

The Hopf fibrations \cite{Hop}
$$
S^0\to{}S^n\to{}\RP^n,
\qquad
S^1\to{}S^{2n+1}\to{}\CP^n,
\qquad
S^3\to{}S^{4n+3}\to{}\HP^n,
\qquad
S^7\to{}S^{15}\to{}S^8
$$
provide fibrations of spheres whose fibers are \textit{great spheres}.
Algebraic topology imposes severe restrictions on possible dimensions of the spheres and the fibers,
the above list actually contains all the existing cases.
The study of such fibrations is motivated, in particular,
by the classical Blaschke conjecture of differential geometry,
see \cite{GW}-\cite{GWZ2}, \cite{Yan}, \cite{Mc} and \cite{Sal} for classification of fibrations of
spheres by great spheres up to diffeomorphism.

Given a fibration of $S^n$ by great spheres $S^p$, the radial projection from the center
on an affine hyperplane yields a fibration of $\R^n$ by \textit{pairwise skew} $p$-planes.
Two affine subspaces of an affine space are called skew if they neither intersect nor contain
parallel directions.
For example, the projection of the Hopf fibration $S^1\to{}S^3\to{}S^2$ gives
a fibration of $\R^3$ by pairwise skew straight lines 
(that lie on a nested family of hyperboloids of one sheet), see Figure \ref{FirstF}.

\begin{figure}[hbtp]
\includegraphics[width=6cm]{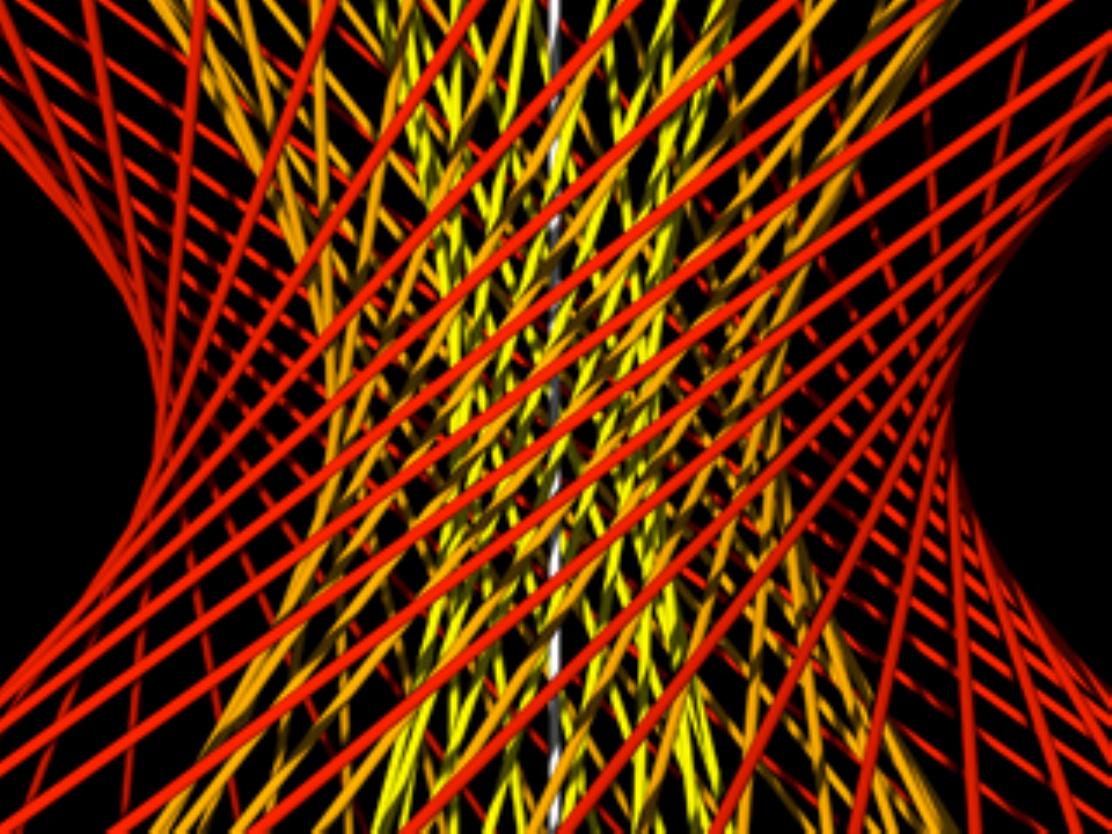}
\caption{The figure is due to David Eppstein, see the Wikipedia article
``Skew lines''.}
\label{FirstF}
\end{figure}

In this paper, we study fibrations
$$
\R^p\to\R^n\to\R^q
$$
whose fibers are pairwise skew affine subspaces.
We refer to such fibrations as $(p,n)$-\textit{fibrations}.
The topological restrictions are less prohibitive in this situation, 
and the list is considerably longer that that of Hopf.
Our goal here is to describe all pairs $(p,n)$ for which $(p,n)$-fibrations exist.
The problem of classification is not addressed here.

Let us consider an a-priori less restrictive problem:
given $p$ and $n$, does there exist a fibration of a small disc, $D^n$,
with pairwise skew flat $p$-dimensional fibers, i.e., by intersections of $D^n$ with
pairwise skew affine subspaces?
We call such fibrations \textit{local}.
A local fibration may fail to extend to the whole $\R^n$.
Yet, both problems have the same answer.

\begin{thm}
\label{TheThm}
$\R^n$ admits a continuous fibration with $p$-dimensional pairwise skew affine fibers
if and only if 
\begin{equation}
\label{TheEq}
p\leq\r(n-p)-1
\end{equation}
where $\r$ is the classical Hurwitz-Radon function.
The local problem has the same answer.
\end{thm}

We define the Hurwitz-Radon function in the next section.

We discuss various particular cases and consequences of this theorem
in Section \ref{ExSec}.
Let us mention now just one.

We have an obvious reduction from a local $(p,n)$-fibration to
a local  $(p-1,n-1)$-fibration
(intersection with a hyperplane transverse to the fibers).
This leads to the notion of \textit{dominant} $(p,n)$-fibration
such that $(p+1,n+1)$-fibration does not exist.
Hopf's $(1,3),\,(3,7)$ and $(7,15)$-fibrations are dominant.
The remarkable fibration
$$
\R^8\to\R^{24}\to\R^{16}
$$
is the first dominant non-Hopf fibration.

Let us briefly mention a related subject,
totally skew embeddings of manifolds into Euclidean space.
A submanifold is called totally skew if the tangent spaces at any pair
of distinct points are skew, see
\cite{GT}, \cite{ST}.

\section{Hurwitz-Radon function and square identities}

The Hurwitz-Radon function $\r(q)$ is defined as follows.
Write $q=2^k(2m+1)$, then
$$
\r(q)=\left\{ \begin{array}{lcll}
2k+1, \quad&k\equiv& 0 &\mod 4\\
2k, \quad&k\equiv& 1,2 &\mod 4\\
2k+2, & k\equiv &3 &\mod 4.
\end{array}
\right.
$$
This formula was discovered in about 1920 independently by
A. Hurwitz and J. Radon in the following context.

A square identity of size $[r,s,q]$ is an identity
$$
(a_1^2+\cdots{}+a_r^2)\,(b_1^2+\cdots{}+b_s^2)
=c_1^2+\cdots{}+c_q^2,
$$
where $c_i$ are bilinear expressions in  $a_i$ and $b_i$ with
real coefficients.
The Hurwitz-Radon theorem states that  $r=\r(q)$ is the largest integer 
for which there exists a $[r,q,q]$-identity \cite{Hur}, \cite{Rad}. 

The cases $q=1,2,4$ and $8$ are special: these are the only case where $\r(q)=q$.
For example Euler's $4$ square identity reads
$$
\begin{array}{rcl}
(a_1^2+a_2^2+a_3^2+a_4^2)\,(b_1^2+b_2^2+b_3^2+b_4^2)&=&
(a_1b_1-a_2b_2-a_3b_3-a_4b_4)^2\\[4pt]
&+&(a_1b_2+a_2b_1+a_3b_4-a_4b_3)^2\\[4pt]
&+&(a_1b_3+a_3b_1-a_2b_4+a_4b_2)^2\\[4pt]
&+&(a_1b_4+a_4b_1+a_2b_3-a_3b_2)^2.
\end{array}
$$
This corresponds to multiplication of quaternions.
These four special cases correspond to the classification of normed division algebras,
$\R,\C,\bbH,\bbO$,
also due to Hurwitz \cite{Hur1}.

We refer here to \cite{Sha} for general information about square identities and related topics.
The subject still remains an active area of research; see, e.g., \cite{MGO}, \cite{LMO}.

\subsection{Hurwitz matrices}

Hurwitz in his last paper \cite{Hur}  (and independently Radon \cite{Rad})
considers square identities of type $[r,q,q]$.
He rewrites these identities in the form of 
$r$-dimensional families of $(q\times{}q)$-matrices.

Consider an identity 
$$
\left(a_1^2+\cdots{}+a_r^2 \right) \left(b_1^2+\cdots{}+b_q^2 \right)
=c_1^2+\cdots{}+c_q^2.
$$
Since $c_i$ are bilinear in $a_j,b_k$, one has
$$
c=(a_1A_1+\cdots+a_rA_r)\,b
$$
where $b$ and $c$ are $q$-vectors and $A_i$ are real $(q\times{}q)$-matrices.
The square identity is equivalent to the equations: 
\begin{equation}
\label{Hur}
A^{t}_i\, A_i=\Id,
\qquad
A^{t}_i\, A_j+A^{t}_j\, A_i=0,
\quad i\not=j.
\end{equation}
Thus, every $r$-dimensional space of real $(q\times{}q)$-matrices
with this property defines a square identity of type $[r,q,q]$.

It follows that every non-trivial linear combination of the matrices $A_1,\ldots,A_r$
is non-degenerate.
Indeed, if $A=a_1A_1+\cdots+a_rA_r$ then $A^{t}A=\left(a_1^2+\cdots+a_r^2\right)\Id$.

\subsection{Vector fields on spheres}

Square identities have several geometric applications.
Perhaps the best known one is the application to vector fields on spheres.
The famous theorem of F.~Adams \cite{Ada} states that,
{\it on a sphere $S^{q-1}$, the maximal number of vector fields 
linearly independent at each point is $\r(q)-1$}.

The existence part follows from the Hurwitz construction.
Given a square identity of type $[r,q,q]$, i.e., a set of matrices $A_1,\ldots,A_r$
satisfying (\ref{Hur}), following Hurwitz, consider $r-1$ matrices
$$
A_r^{-1}A_1,\,\ldots\,,A_r^{-1}A_{r-1}
$$
that are skew-symmetric.
Applying these linear maps to the points of a round sphere $S^{q-1}\subset\R^q$,
one obtains $r-1$ (linear) tangent vectors fields.
The condition (\ref{Hur}) guarantees that these vector fields are indeed linearly
independent at each point.

The converse statement is the difficult part of the Adams theorem
whose proof is based on algebraic topology.

\subsection{Property P}

F. Adams, P. Lax and R. Phillips \cite{ALP} considered
$r$-dimensional spaces of real $(q\times{}q)$-matrices
with the property that every non-zero element of the space is a non-degenerate matrix;
they called it ``property P''.
This property is a-priori weaker than property (\ref{Hur}).
It follows from the Adams theorem that the restriction to the dimension~$r$ is actually the same:
$r\leq\r(q)$.

We extend the notion of property P to rectangular matrices in a straightforward way.
A vector space of $(p\times{}q)$-matrices has property P if every non-zero matrix of this space
has rank $\min(p,q)$.

\section{Examples and corollaries}\label{ExSec}

\subsection{Tables for small $n$}

The values of dimensions $p$
(as a function in~$n$) of the fibers of $(p,n)$-fibrations
are presented in the following table.
$$
\setlength{\extrarowheight}{3pt}
\begin{array}{c||c|c|c|c|c|c|c|c|c|c|c|c|c|c|c|c|c|c|c|c|c|c}
n  &{\bf 3} &4 &5 &6 &{\bf 7} &\!8 \!& 9&10 & 11 & 12&13&14&{\bf 15}
&\!16\!&17&18&19&20&21&22&23&{\bf 24}\\
\hline
\hline
p & {\bf 1} & & 1 & 2 &{\bf 3} &  &1 & 2 & 3      & 4 &5&6&{\bf 7} & &1&2&3       &4&5&6&7        &{\bf 8}\\
   &             & &    &   &{\bf 1}  &  &    &   &{\bf 1}&   & 1&2&{\bf 3}&   &   &  &{\bf 1}&  &1&2&{\bf 3}&       \\
   &&&&&&&&&&&&&{\bf 1}&&&&&&&&{\bf 1}
\end{array}
$$ 
The dominant values are in boldface.
The dominant values for $n=3,\,7,\,15$ correspond to the Hopf fibrations.
The $(8,24)$-fibration is non-Hopf, in particular, it dominates the Hopf $(3,19)$-fibration.
The next values are as follows
$$
\setlength{\extrarowheight}{3pt}
\begin{array}{c|c|c|c|c|c|c|c|c|c|c|c|c|c|c|c|c|c|c|c}
25 &26&27&28 &29&30 &{\bf 31}&\!32\!&33& 34&35&36&37&38
&39&40&{\bf 41}&42&43&44\\
\hline
\hline
1 & 2 & 3   & 4 & 5& 6 & {\bf 7} &&1& 2 &3       &4  &5&6 & 7        &8&{\bf 9}&2 & 3&4 \\
  &  & {\bf 1}&   & 1& 2 &{\bf 3}&   &&     &{\bf 1}& &1   & 2 &{\bf 3}&  &{\bf 1}&   &{\bf 1}   \\
   &&&&&&{\bf 1}&&&&&&&&{\bf 1}
\end{array}
$$ 
and also
$$
\setlength{\extrarowheight}{3pt}
\begin{array}{c|c|c|c|c|c|c|c|c|c|c|c|c|c|c|c|c|c|c|c}
45 &46&{\bf 47}&\!48\! &49&50 &51&52&53& 54&55&{\bf 56}&57&58
&59&60&61&62&{\bf 63}&\!64\!\\
\hline
\hline
5&6 & {\bf 7 }&  & 1&2& 3    &4 & 5&6&7        &{\bf 8}&1 &2&3       &4&5&6 &{\bf 7}&  \\
 1&2& {\bf 3} &   & & &{\bf 1}&   &1 &2&{\bf 3}&          &   &   &{\bf 1}&  &1&2&{\bf 3}     \\
   &   & {\bf 1}&&&&&&&&{\bf 1}&&&&&&&&{\bf 1}
\end{array}
$$ 
together with
$$
\setlength{\extrarowheight}{3pt}
\begin{array}{c|c|c|c|c|c|c|c|c|c|c|c|c|c|c|c|c|c|c|c}
65 &66&67&68 &69&70 &71&72&73&74&{\bf 75}&76&77&78
&{\bf 79}&\!80\!&81&82&83&84\\
\hline
\hline
1&2& 3& 4& 5&6& 7      &8  &9&10&{\bf 11}&4&5&6 &{\bf 7}&  &1&2&3  & 4  \\
  &  & 1&   & 1 &2&{\bf 3}&   &1&2  &{\bf 3}  &    &1&2&{\bf 3}&  &  &   &{\bf 1}     \\
   &   & &&&&{\bf 1}&&&&{\bf 1}&&&&{\bf 1}&&&& 
\end{array}
$$ 
We note the most interesting cases in sight $(8,24)$, $(9,41)$ and $(11,75)$.

\subsection{Vertical heredity}

We see from the table that some columns contain more that one entry.
For example, $\R^{75}$ admits skew fibrations with fibers of dimensions $11,3$ and $1$.
These triples appear earlier in the table: $\R^{11}$ admits skew fibrations 
with fibers of dimensions $3$ and $1$.
This is not a coincidence.

In one direction, this is due to the following obvious fact:
given a $(p_1,n)$-fibration and a $(p_2,p_1)$-fibration, one obtains a $(p_2,n)$-fibration.
The converse statement is less obvious.

\begin{prop}
\label{Matreshka}
Given a $(p_1,n)$-fibration and a $(p_2,n)$-fibration with $p_1>p_2$, there exists a
$(p_2,p_1)$-fibration.
\end{prop}

\noindent
In other words, if $p_1$ and $p_2$ appear in $n$-th column of the table,
then $p_2$ appears in $p_1$-th column.
In this sense, the first row of the table contains the full information.

\subsection{Dominant cases}

The dominant cases determine all the entries in the table.
Each dominant $(p,n)$-fibration generates a sequences
$(p,n),\,(p-1,n-1),\,(p-2,n-2),\ldots$
reaching the previous dominant case where $p$ jumps up again.

Let us start with a simple observation:
a fibration $(p,n)$ is dominant if and only if there exists $q$ such that
$$
p=\r(q)-1,
\qquad
n=q+\r(q)-1.
$$
Indeed, in this case $(p+1,n+1)$ does not satisfy (\ref{TheEq}) so that the $(p,n)$-fibration is dominant.
Conversely, if $p<\r(q)-1$, then we can increase $p$ and $n$ by 1, and the $(p,n)$-fibration is not dominant.

The most interesting dominant fibrations are those with maximal $p$,
i.e., the fibrations that appear in the first row of the above table.
Let us call these fibrations \textit{doubly dominant}.
The following statements are suggested by the above table.

\begin{prop}
\label{DomProp}
(i)
For every $k$, the fibration $(p,n)$ with
$$
p=\r(2^k)-1,\qquad
p=2^k+\r(2^k)-1
$$
is doubly dominant.

(ii)
If a $(p,n)$-fibration is doubly dominant and $n\geq8$, then
$q=n-p\equiv0\,(\!\!\!\mod8)$.
\end{prop}

\subsection{Extreme cases}

Let us consider some special cases.
The first one is the case where no fibrations with  pairwise skew affine fibers exist.

\begin{prop}
\label{TheCor}
(i)
For every $k$, space $\R^{2^k}$ does not admit fibrations with  pairwise skew affine fibers.

(ii)
For every odd $m$, there exists $k_0$ such that if $k\geq{}k_0$ then
space $\R^{2^km}$ does not admit fibrations with  pairwise skew affine fibers.
If $m=3$ or $5$ then $k_0=4$.

(iii)
If $\R^{n}$ does not admit fibrations with pairwise skew affine fibers, then
$n\in\{1,2,4,8\}$ or $n\equiv0\,(\!\!\mod16)$.
\end{prop}

It follows from the definition of a skew fibration that $p\leq{}q-1$.
The second extreme case is described in the next statement.

\begin{prop}
\label{TheCorBis}
$p=q-1$ if and only if $n\in\{1,\,3,\,7,\,15\}$.
\end{prop}

These statements will be deduced from Theorem \ref{TheThm} in 
Section \ref{Sec4}.
Let us emphasize that the statements rely on the full force of Adams's theorem
and cannot be obtained using less advanced topological methods
(such as Stiefel-Whitney classes).

\section{Proofs}\label{Sec4}

\subsection{Proof of Theorem \ref{TheThm}} 

Suppose a local $(p,n)$-fibration is given.
Choose one fiber and identify it with $\R^p\subset\R^n$.
Consider the subspace $\R^q$ orthogonal to $\R^p$.
Every fiber close to the ``horizontal'' $\R^p$ is the graph of an affine map:
$$
\eta=B(y)\,\xi+y,
$$
where $y$ is the coordinate in the transversal $\R^q$
and where $B(y):\R^p\to\R^q$ is a linear map continuously depending on $y$
(defined for $y$ sufficiently close to the origin).

Define the linear map $A(y):\R^{p+1}\to\R^q$ whose matrix
is obtained from the matrix of~$B(y)$ by joining the column $y$.
The next lemma gives a necessary and sufficient condition for the fibration
to be skew.

\begin{lem}
\label{LemCond}
The fibers of the fibration are pairwise skew if and only if
$\ker(A(y_1)-A(y_2))=0$ for all pairs of distinct $y_1,y_2$ in the domain.
\end{lem}

\begin{proof}
The fibers through $y_1$ and $y_2$ intersect if and only if the system
$$
\left\{
\begin{array}{rcl}
\eta&=&B(y_1)\,\xi+y_1\\[4pt]
\eta&=&B(y_2)\,\xi+y_2
\end{array}
\right.
$$
has a solution.
This is equivalent to
$$
\left(A(y_1)-A(y_2)\right)
\Big(\!\!
\begin{array}{c}
\xi\\
1
\end{array}\!\!
\Big)=0.
$$
Likewise, the fibers through $y_1$ and $y_2$ contain parallel directions
if and only if the system
$$
\left\{
\begin{array}{rcl}
\eta&=&B(y_1)\,\xi\\[4pt]
\eta&=&B(y_2)\,\xi
\end{array}
\right.
$$
has a non-zero solution.
This is equivalent to
$$
\left(A(y_1)-A(y_2)\right)
\Big(\!\!
\begin{array}{c}
\xi\\
0
\end{array}\!\!
\Big)=0.
$$
Hence the result.
\end{proof}

Given a local $(p,n)$-fibration,
let us construct $p$ linearly independent tangent vector fields on $S^{q-1}$.
Consider a small sphere $|y|=\e$.
At a point $y$ of the sphere, consider the vectors given by the columns of~$B(y)$,
and project them on the tangent plane $T_yS^{q-1}$.
If these projected vectors are linearly dependent, then so are the columns of $A(y)$.
This contradicts Lemma \ref{LemCond} for points $y$ and~$0$.

According to the Adams theorem \cite{Ada}, one has the inequality $p\leq\r(q)-1$.
This proves necessity.

Suppose now that $p\leq\r(q)-1$.
The Hurwitz-Radon theorem implies the existence of matrices
$A_1,\ldots,A_{p+1}$ that span a space of $(q\times{}q)$-matrices with property P.
In other words, we have a bilinear map
$$
A:\R^{p+1}\times\R^q\to\R^q
$$
such that $A(x,y)\not=0$, if $x\not=0$ and $y\not=0$.
Equivalently, we have a $q$-dimensional space of $\left((p+1)\times{}q\right)$-matrices
$A(.,y)$ with property P.

Consider the linear map of this $q$-dimensional space of $\left((p+1)\times{}q\right)$-matrices 
to $\R^q$ sending a matrix to its last column.
This linear map is a linear injection, otherwise property P is violated.
Hence it is a linear isomorphism and the last column can be chosen as the coordinate $y$.
We thus obtain a linear family $A(y):\R^{p+1}\to\R^q$.
The condition of Lemma \ref{LemCond} is satisfied because 
$A(y_1)-A(y_2)=A(y_1-y_2)$ by linearity.

It remains to show that the defined fibration is global.
This means that for every $(\xi,\eta)\in\R^n$ there exists $y\in\R^q$
such that $\eta=B(y)\,\xi+y$, i.e.,
\begin{equation}
\label{TE}
\eta=
A(y)
\Big(\!\!
\begin{array}{c}
\xi\\
1
\end{array}\!\!
\Big).
\end{equation}
Recall that the bilinear map $A(x,y)$ has the property that 
the linear map $A(x,.):\R^q\to\R^q$ is a linear isomorphism for $x\not=0$,
in particular, for $x=\Big(\!\!
\begin{array}{c}
\xi\\
1
\end{array}\!\!
\Big)$.
Therefore, equation (\ref{TE}) has a solution.

\subsection{Proof of Proposition \ref{Matreshka}}

Let $n=p_1+q_1=p_2+q_2$ and $p_1>p_2$.
Arguing by contradiction, assume that there is no $(p_2,p_1)$-fibration.
Theorem \ref{TheThm} implies:
$$
p_1<\r(q_1),
\qquad
p_2<\r(q_2),
\qquad
p_2\geq\r(p_1-p_2).
$$
Write $p_1-p_2=2^km$ where $m$ is odd, then $p_2\geq\r(2^k)$.
On the other hand, $q_2-q_1=2^km$.
It follows that $q_1$ and $q_2$ cannot be both divisible by $2^{k+1}$,
and thus
$$
\r(q_1)\leq\r(2^k)
\qquad
\hbox{or}
\qquad
\r(q_2)\leq\r(2^k).
$$
Consider the first case.
Then $p_2\geq\r(2^k)\geq\r(q_1)>p_1$, a contradiction.
In the second case, one has $p_2\geq\r(2^k)\geq\r(q_2)>p_2$, again a contradiction.

\subsection{Proof of Proposition \ref{DomProp}}\label{DoP}

Part (i).
Given a dominant fibration $(p,n)$ with $p=\r(2^k)-1$ and $n=2^k+\r(2^k)-1$, we want to show that
there is no $p'>p$ and $q'$ such that $p'+q'=n$ and $p'\leq\r(q')-1$.
This is true since $q'<q$ so $\r(q')<\r(q)$ which contradicts the inequality $p'>p$.

Part (ii).
Assume that $q$ is odd so that $p=0$.
Let $q'=q-1$ then $p'=1$ and $\r(q')\geq2$ hence $p'<\r(q')$.
Thus $(p,n)$-fibration is not doubly dominant.

Assume that $q=2m$, where $m$ is odd so that $p\leq1$.
Let $q'=q-2$, then $p'=p+2\leq3$.
Since $q'\equiv0\,(\!\!\!\mod4)$, one has $\r(q')\geq4$, and thus $p'\leq\r(q')-1$.

Assume that $q=4m$, where $m$ is odd so that $p\leq3$.
Let $q'=q-4$, then $p'=p+4\leq7$.
Since $q'\equiv0\,(\!\!\!\mod8)$, one has $\r(q')\geq8$, and thus $p'\leq\r(q')-1$.

\subsection{Proof of Proposition \ref{TheCor}}\label{CoS}
Part (i).
Assume that $n=2^k$ and let $p=2^\ell\,t$ where $t$ is odd.
Then 
$$
q=2^k-2^\ell\,t=2^\ell(2^{k-\ell}-t).
$$
Therefore, the condition $p\leq\r(q)-1$ reads
$2^\ell\,t\leq\r(2^\ell)-1$.
Since $t\geq1$, we have, in particular, the inequality
$2^\ell<\r(2^\ell)$
 that never holds.
 
 Part (ii).
 Assume that $n=2^k\,m$ where $m$ is odd, and let $p=2^\ell\,t$ where $t$ is odd.
Then one has three cases: a) $r<k$, b) $r>k$, c) $r=k$.
In cases a) and b) the argument is similar to the one above.
In case c) one has
$
q=2^k(m-t).
$
The condition (\ref{TheEq}) then reads
$$
2^k\,\leq\,2^kt\,(=p)\,<\,\r\left(2^k(m-t)\right).
$$
Since $m-t=2^\ell\,s$ where $\ell$ is bounded above by a constant depending on $m$ and $s$ is odd,
one has $2^k<\r(2^{k+\ell})\leq2(k+\ell)+2.$
This inequality fails for sufficiently large $k$.

Part (iii).
If $n$ is odd then set $p=1$. Since $\r(n-1)\geq2$,
Theorem \ref{TheThm} implies the existence of a fibration $\R\to\R^n\to\R^{n-1}$.

Assume that $n=2m$, where $m\geq3$ is odd.
Set $p=2$, then $q=2(m-1)\equiv0\,(\!\!\!\mod4)$, and therefore $\r(q)\geq4$.
This implies the existence of a $(2,n)$-filtration.

Assume that $n=4m$, where $m\geq3$ is odd.
Set $p=4$, then $q=4(m-1)\equiv0\,(\!\!\!\mod8)$, and therefore $\r(q)\geq8$.
This implies the existence of a $(4,n)$-filtration.

Assume finally that $n=8m$, where $m\geq3$ is odd.
Set $p=8$, then $q=8(m-1)\equiv0\allowbreak\,(\!\!\!\mod16)$, and therefore $\r(q)>8$.
This implies the existence of a $(8,n)$-filtration.
This completes the proof.

\subsection{Proof of Proposition \ref{TheCorBis}}\label{CoSBis}

Assume that $p=q-1$, then the condition (\ref{TheEq}) implies
that $q\leq\r(q)$.
This holds only for $q=1,2,4,8$.

\medskip

\noindent \textbf{Acknowledgments}.
This project originated at the Institut Henry Poincar\'e (IHP), in the framework of the RiP program. 
We are grateful to IHP for its hospitality.
We are also pleased to thank S.~Morier-Genoud for enlightening discussions.
V.~O. was partially supported by the PICS ``PENTAFRIZ'' of CNRS,
S.~T. was partially supported by the Simons Foundation grant No 209361 and by the NSF grant DMS-1105442.


\end{document}